\documentclass[a4paper,12pt]{article}
\usepackage[T1]{fontenc}
\usepackage{amsfonts,amsthm,amssymb,mathrsfs,amsmath}
\usepackage{graphicx,graphics}
\usepackage{relsize}
\usepackage[all]{xy}
\usepackage{tikz}
\usepackage{color}
\usepackage{showlabels}
\allowdisplaybreaks

\newcommand{\ZZ}{\mathbb{Z}}
\newcommand{\NN}{\mathbb{N}}

\newcommand{\CC}{\mathbb{C}}
\newcommand{\RR}{\mathbb{R}}

\newcommand{\dprod}[2]{\left\langle #1,#2\right\rangle}

\newcommand{\supp}{\mathop{\mathrm{supp}}}

\newtheorem{theorem}{Theorem}[section]

\newtheorem{remark}[theorem]{Remark}
\numberwithin{equation}{section}
\newtheorem{definition}[theorem]{Definition}
\newtheorem{example}[theorem]{Example}

\title{ On multiplicity of spectrum for Anderson type operators with higher rank perturbations}
\author{Anish Mallick\footnote{e-mail: anish.mallick@icts.res.in, Institute: ICTS Bangalore, India.} ~~\& 
Narayanan P A \footnote{e-mail: panarayanan.pothani@gmail.com, Institute: IMSc Chennai, India.}}
\date{\today}

\begin{document}
\setlength{\belowdisplayskip}{2pt} \setlength{\belowdisplayshortskip}{2pt}
\setlength{\abovedisplayskip}{2pt} \setlength{\abovedisplayshortskip}{2pt}
\maketitle
\begin{abstract}
Here, we focus on Anderson type operators over infinite graphs where the randomness acts through higher rank perturbations.
We show that for special family of graphs, the operator has non-trivial multiplicity for its pure point spectrum. 
We, also, show that for some family of graphs, 
any unitary which fixes the random operator, arising from an automorphism of the graph is identity; but that, for these graphs the spectrum of the random operator has non-trivial multiplicity.
\end{abstract}
\section{Introduction}
The theory of random operators have gained a significant attention over the last few decades.
The Anderson tight binding model is an example of random operator which was developed by P. W. Anderson \cite{PA} to study the transport property of spin waves on doped semi-conductor. 
Many works focus on the spectrum of this operator.
Under different settings, the existence of the pure point and the absolute continuous spectra are proved; see \cite{K2,CL,PH} for a comprehensive review on this topic.
%
There are a few other families of random operators on infinite dimensional Hilbert space, for example, random Schr\"{o}dinger operator, random Landau Hamiltonian and random dimer/polymer model.
Many results from the theory of Anderson tight binding model extend to these models as well. 
On the other hand, some results which are true for Anderson tight binding model may not hold for these models.
One such example is multiplicity of the spectrum. 
This work focuses on Anderson type operator with non-trivial multiplicity.

The multiplicity problem in the case of Anderson tight binding model has been investigated in a few works.
For example, Barry Simon \cite{BS2} (works of other authors include \cite{KM2,JL2}) showed that the spectrum is simple in the region of localization for Anderson tight binding Hamiltonian. 
Jak\v{s}i\'{c}-Last \cite{JL1,JL2} showed that for Anderson type operators where the randomness acts through rank one perturbations, the singular spectrum is always simple.

Models where randomness acts through higher rank operators have been considered, as well.
In the case of higher rank perturbations, at least for general Anderson type operators, the best one can provide are bounds on the multiplicity, and
it does not exclude the possibility that 
in some special cases simplicity may show up.
Some works dealing with cases where randomness acts through higher rank perturbations are  \cite{DE,NNS,SSH} where the authors showed simplicity of pure-point spectrum.
Many of these results are inspired from the heuristics which states that \emph{multiplicity of a Hamiltonian arises from symmetry of the underlying problem}. 
So, for Anderson type operators where none of the symmetries of the underlying space (for example in the case of Anderson tight binding model, 
these symmetries will be translation by the action of $\ZZ^d$) keep the random Hamiltonian invariant, the spectrum (at least, the pure point) should be simple.
Though in \cite{SSH} Sadel and Schulz-Baldes showed that absolute continuous spectrum can have non-trivial multiplicity.
For general Anderson type operators, it is possible (see \cite{AD1,AM1}) to provide  bounds (based on the Green's function) on the multiplicity for the singular spectrum.
In some special cases, this bound may imply simplicity of the singular spectrum,  see \cite{AM3,AD1}.
One of the goals of this work is to show that the above mentioned heuristics does not hold in its strictest sense. 
In Section \ref{sec3}, we provide a family of Anderson type operators for which the multiplicity of pure point spectrum is high,
but the multiplicity does not arise from any symmetry of the underlying space.

We should explain the terms \emph{symmetry} and \emph{heuristics} within the context of this work.
We will be working with Anderson type operators over graphs, so by \emph{symmetry of the underlying space} we mean graph automorphism.
So, on the Hilbert space over the graph, we can use the automorphism to construct unitary operators.
We will show that, there are Anderson type operators over certain graphs, for which the point spectrum has non-trivial multiplicity,  even though
any unitary arising from automorphism which fixes the operator is identity.

Many works involving local eigenvalue statistics for higher rank Anderson type operators, for example \cite{npa,HK1}, showed that the statistics is compound Poisson.
But, that itself does not remove the possibility that the statistics is simple Poisson.
The operator discussed in Section \ref{sec2} shows up as the limiting operator obtained in the work \cite{npa}. 
Hence, the local eigenvalue statistics obtained in \cite{npa} is
a non-trivial compound Poisson.
In a similar fashion, the family of operators from section \ref{sec3} implies that the local eigenvalue statistics defined in the work \cite{HK1}
can be a non-trivial compound Poisson point process (i.e., the support of the L\'{e}vy measure has multiple points in it).

In the section \ref{sec2}, we show that the Anderson operator on canopy tree with higher rank perturbations has non-trivial multiplicity 
depending on the rank of the perturbations and the degree of (any vertex which are away from the boundary of) the graph.
In section \ref{sec3}, we construct a family of Anderson type operators which are ergodic under a group action and which has non-trivial multiplicity.
We, also, classify all the unitaries arising from automorphisms of the graph which fixes the operator.
As a corollary, we show that there are graphs such that the Anderson operator defined has non-trivial multiplicity and the multiplicity does not arise from any automorphism of the graph.

\section{Canopy Tree}\label{sec2}
In this section we will focus on an infinite canopy tree of degree $K+1$.  
Before going into the definition of the graph, let us establish a convention that will be used. 
An undirected graph $\mathcal{H}$ is a pair of sets $(\mathcal{V},\mathcal{E})$ where $\mathcal{V}$ denotes the set of vertices and $\mathcal{E}$ denotes the set of edges.  An edge $e\in\mathcal{E}$ is viewed as a subset of $\mathcal{V}$ with two elements. We will work with graphs which does not have self-loop, so we can view an edge as a set of cardinality two.

The graph under consideration here, will have $K+1$ neighbors for each vertex except for the leaf nodes. For the proof of the theorem \ref{thmCan1} to work we will set $K>2$.

\begin{definition}
A canopy tree $\mathcal{T}$ of degree $K+1$ is given by the pair $(\mathcal{V},\mathcal{E})$, where the vertex set is $\mathcal{V}=\ZZ\times(\NN\cup\{0\})$ and the edge set is 
$$\mathcal{E}=\left\{\left\{(x,n),\left(\left\lfloor\frac{x}{K}\right\rfloor,n+1\right)\right\}:x\in\ZZ,n\in\NN\cup\{0\}\right\}.$$
\end{definition}

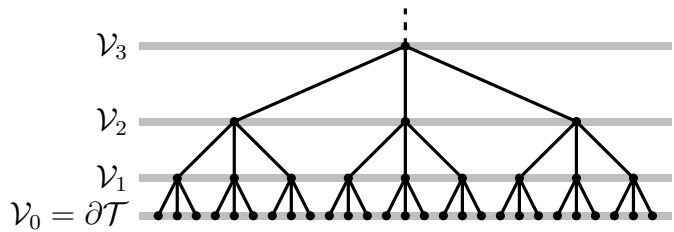
\begin{figure}[ht]
\begin{center}
\begin{tikzpicture}[scale=0.5]
\begin{scope}[rotate=90]
\fill[fill=gray!50](-0.1,-7.0)rectangle(0.1,7.0);
\draw(0,7.0)node[left]{$\mathcal{V}_0=\partial\mathcal{T}$};
\fill[fill=gray!50](0.9,-7.0)rectangle(1.1,7.0);
\draw(1,7.0)node[left]{$\mathcal{V}_1$};
\fill[fill=gray!50](2.4,-7.0)rectangle(2.6,7.0);
\draw(2.5,7.0)node[left]{$\mathcal{V}_2$};
\fill[fill=gray!50](4.4,-7.0)rectangle(4.6,7.0);
\draw(4.5,7.0)node[left]{$\mathcal{V}_3$};

\foreach \x in {-13,...,13}
{
	\filldraw(0,0.5*\x)circle(3pt);
}

\foreach \x in {-4,...,4}
{
	\filldraw(1,1.5*\x)circle(3pt);
    \foreach \y in {1.5*\x-0.5,1.5*\x,1.5*\x+0.5}
    \draw[very thick](1,1.5*\x)--(0,\y);
}

\foreach \x in {-1,...,1}
{
	\filldraw(2.5,4.5*\x)circle(3pt);
    \foreach \y in {4.5*\x-1.5,4.5*\x,4.5*\x+1.5}
    \draw[very thick](2.5,4.5*\x)--(1,\y);
}

\filldraw(4.5,0)circle(3pt);
\draw[very thick](4.5,0)--(2.5,0);
\draw[very thick](4.5,0)--(2.5,4.5);
\draw[very thick](4.5,0)--(2.5,-4.5);

\draw[dashed,very thick](4.5,0)--(5.5,0);
\end{scope}
\end{tikzpicture}
\caption{An example of canopy tree for $K=3$.}
\end{center}
\end{figure}

\noindent We will denote the boundary of the tree by
$$\partial\mathcal{T}=\{(y,0):y\in\ZZ\},$$
and for any $i\in\NN\cup\{0\}$, the set of vertices which are $i$ distance away from the boundary by
$$\mathcal{V}_i=\{(y,i): y\in\ZZ\}.$$
On  $\mathcal{V}$, we denote by $d$ the usual metric of the graph. That is, for any two vertices $ v,w$ in $\mathcal{V} $,   $d(v,w)$ is the length of the shortest path connecting $v$ and $w$.  
We will also need a binary relation $\prec$ on $\mathcal{V}$ which is defined by
$$v\prec w~ \Leftrightarrow ~d(v,\partial\mathcal{T})\leq d(w,\partial\mathcal{T})~\&~d(v,w)=d(w,\partial\mathcal{T})-d(v,\partial\mathcal{T}),$$
where $d(v,\partial{\mathcal{T}}) $ is the distance of $v$ from the boundary. Thus, $v \prec w$ means that $v$ lies in the shortest path between $w$ and the boundary $\partial\mathcal{T}$. 
For $w\in\mathcal{V}$,  the forward neighbor set is defined by
$$N_w=\{v\in\mathcal{V}:v\prec w~\&~d(v,w)=1\}.$$
Note that $N_w$  is empty for $w\in\partial\mathcal{T}$, but for any other vertex it has cardinality $K$.
Finally for $w\in\mathcal{V}$ and $l\in\NN$, we will denote the tree
$$\Lambda_l(w):=\{v\in \mathcal{V}: v\prec w, d(v,w)\leq l\},$$
where the edges are obtained by restricting the edges of $\mathcal{T}$ to $\Lambda_l(w)$.

The random operator of interest is defined on the Hilbert space $\ell^2(\mathcal{T})$.
Denote by $\Delta_{\mathcal{T}}$ to be the adjacency operator on $\ell^2(\mathcal{T})$, defined by
$$(\Delta_{\mathcal{T}} u)(v)=\sum_{d(v,w)=1}u(w)\qquad\forall v\in\mathcal{V},$$
and the projection $P_S$,  for $S\subset \mathcal{V}$, by
$$(P_S u)(v)=\left\{\begin{matrix} u(v) & v\in S\\ 0 & v\not\in S\end{matrix}\right.\qquad \forall v\in\mathcal{V},$$
for any $u\in\ell^2(\mathcal{T})$.
%
%
The family of random operators in consideration is given by
\begin{equation}\label{canOpEq1}
H^\omega_{\mathcal{T}}=\Delta_{\mathcal{T}}+\sum_{x\in \mathcal{N}} \omega_x P_{\Lambda_l(x)},
\end{equation}
for some $l\in\NN$, where 
$$\mathcal{N}=\bigcup_{m\in\NN\cup\{0\}} \mathcal{V}_{m(l+1)+l},$$
and $\{\omega_x\}_{x\in\mathcal{N}}$ are independent identically distributed random variables. 

We will denote $\Delta_n$ to be the adjacency matrix for the tree $\Lambda_{n}(x)$, for $x\in\mathcal{V}_{n}$, for $n\in\NN$
(since all of these trees are isomorphic, we do not need to specify the root other than the distance from boundary). 
\begin{theorem}\label{thmCan1}
For $K>2$, let $\mathcal{T}$ denote the canopy tree of degree $K+1$ and on the Hilbert space $\ell^2(\mathcal{T})$ define the random operator $H^\omega_{\mathcal{T}}$ by \eqref{canOpEq1}, for some $l\geq 2$.
Set the random variables  $\{\omega_x\}_{x\in\mathcal{N}}$ to be independent and identically distributed following a real absolutely continuous distribution $\mu$. 
Then
$$\sigma(\Delta_{l-1})+\supp(\mu)\subset \sigma_{pp}(H^\omega_{\mathcal{T}})$$
and the maximum multiplicity of point spectrum in $\sigma(\Delta_l)+\supp(\mu)$ is at least $K-1$.
\end{theorem}
\begin{proof}
Let $\mathcal{T}_{l-1}$ denote a tree with root $e$ which is isomorphic to the tree $\Lambda_{l-1}(x)$, for $x\in\mathcal{V}_{l-1}$. 
Using the fact that all the $\Lambda_{l-1}(x)$ are identical for any $x\in\mathcal{V}_{l-1}$, we will denote $\phi_x$ to be the isomorphism 
$$\phi_x:\Lambda_{l-1}(x)\rightarrow \mathcal{T}_{l-1}.$$
We will view $\Delta_{l-1}$ as the adjacency matrix for the graph $\mathcal{T}_{l-1}$. 
Finally, for $E\in\sigma(\Delta_l)$ consider an normalized eigenvector $\psi$ corresponding to the eigenvalue $E$. 
\\\\
\noindent{\bf Claim:} {\it For any $x\in\mathcal{V}_l\subset \mathcal{N}$, observe that $E+\omega_x$ is an eigenvalue of the operator $H^\omega_{\mathcal{T}}$ with multiplicity at least $K-1$.}

\noindent To show this, we are going to define the $K-1$ orthonormal eigenvectors for $E+\omega_x$. 
Let $\alpha:=(\alpha_y), y\in N_x$ be an $K-1~$-tuple in $\mathbb{R}^n$, satisfying the following conditions 
\begin{equation}\label{thm1Eq1}
\sum_y \alpha_y=0~~\&~~\sum_y |\alpha_y|^2=1.
\end{equation}

For each such $\alpha$, define the vector $\Psi^{(\alpha)} \in \ell^2(\mathcal{T})$ by
$$\Psi^{(\alpha)}(p)=\left\{ 
\begin{matrix} \alpha_y \psi(\phi_y(p)), &  \text{if }p\prec y ~ \text{ for some }~ y\in N_x \\
0, & ~\text{if}~ p\not\in \cup_{y\in N_x} \Lambda(y)  
\end{matrix} \right.\qquad \forall p\in\mathcal{T},$$

Observe that $\Psi^{(\alpha)}$ satisfies
$$[(H^\omega_{\mathcal{T}}-(E+\omega_x))\Psi^{(\alpha)}](p)=0\qquad\forall p\in\mathcal{V}\setminus \Lambda_l(x)$$
trivially, because all the entries that show up are defined to be zero. 
For any $p\in\Lambda_{l-1}(y)$ where $y\in N_x$, we have
$$[(H^\omega_{\mathcal{T}}-(E+\omega_x))\Psi^{(\alpha)}](p)=\alpha_y[\Delta_{\mathcal{T}_l}\psi](\phi_y(p))-E\psi(\phi_y(p))=0.$$
Here we are using the fact that $\Psi^{(\alpha)}(x)=0$, hence $[\Delta_{\mathcal{T}}\Psi^{(\alpha)}](p)=[\Delta_{\mathcal{T}_l}\psi](\phi_y(p))$.
Finally, at $x$ we have
\begin{align*}
 &[(H^\omega_{\mathcal{T}}-(E+\omega_x))\Psi^{(\alpha)}](x)\\
 &\qquad=\sum_{y\in N_x} \Psi^{(\alpha)}(y)=\psi(e)\sum_{y\in N_x} \alpha_y=0
\end{align*}
by definition of $(\alpha_y)$. Observe that, for any $(\alpha_y)_y$ and $(\beta_y)_y$ that satisfies \eqref{thm1Eq1}, we have
$$\dprod{\Psi^{(\alpha)}}{\Psi^{(\beta)}}_{\ell^2(\mathcal{T})}=\sum_{y\in N_x}\alpha_y\beta_y.$$
Hence we can have $K-1$ orthonormal vectors $\Psi^{(\alpha)}$ which are eigenvectors for $H^\omega_{\mathcal{T}}$ for the eigenvalue $E+\omega_x$. 

Now using the fact that $\{\omega_x\}_{x\in\mathcal{V}_l}$ are i.i.d, we have 
$$\overline{\{E+\omega_x:x\in\mathcal{V}_l\}}=E+\supp(\mu),$$
which completes the proof of the theorem by using the above claim.

\end{proof}
\begin{remark}
Note that, in the theorem we can remove the hypothesis that $\mu$ is absolutely continuous and still the result will hold. 
The only problem is that the set $\sigma(\Delta_{l-1})+\supp(\mu)$ may not have positive Lebesgue measure.
Following the proof, it is easy to see that the measure $\mu_{\Delta_{l-1}}(\cdot)=\sum_{E\in\sigma(\Delta_{l-1})}\mu(\cdot-E)$ is absolutely continuous with respect to density of state measure:
$$N(f)=\lim_{L\rightarrow\infty}\frac{1}{|\Lambda_L(x)|}tr(f(P_{\Lambda_L(x_L)} H^\omega_{\mathcal{T}}P_{\Lambda_L(x_L)} ))\qquad\forall f\in C_c(\RR),$$
where the sequence $x_L\in\mathcal{V}$ is chosen to satisfy $d(x_L,\partial\mathcal{T})=L$ (the limit is non-random follows from\cite{npa}).
So, the density of state measure has non-trivial singular component if $\mu$ is singular.
\end{remark}
\section{Cayley type graph and $G$-ergodic operators}\label{sec3}
It should be noted that in the proof of the Theorem \ref{thmCan1}, the fact that we are working with tree is not important, 
but that there is an eigenvalue of the adjacency matrix for the tree, which has non-trivial multiplicity and there are eigenvectors which are zero at root.
This observation can be used to create other examples of Anderson type operators where similar result holds.

In this section, we focus on a class of  infinite graphs generated by the help of finitely generated groups which are similar to Cayley graph, and define Anderson type operators.
We will show that, under certain circumstances the operator defined has non-trivial multiplicity for its pure point spectrum.
The infinite graphs that we will be working with are defined as follows:
\begin{definition}\label{defCay}
Given a finitely generated group $G$ with generator $g_1,\ldots,g_n$ and 
a set of vertices $v_1,\ldots,v_n,v_{-1},v_{1}\ldots,v_{-n}\in\mathcal{V}$ from a finite undirected graph $\mathcal{H}=(\mathcal{V},\mathcal{E})$,
define the infinite graph $\mathcal{H}_G=(\mathcal{V}_G,\mathcal{E}_G)$ by
\begin{itemize}
 \item The vertex set is given by 
 $$\mathcal{V}_G:=\{(v,g):v\in\mathcal{V},g\in G\},$$
 \item The edge set $\mathcal{E}_G$ is union of the sets
 $$\{\{(v,g),(w,g)\}: \{v,w\}\in\mathcal{E}, g\in G\},$$
 and 
 $$\{\{(v_{-i},g),(v_{i},gg_i)\}: g\in G, 1\leq i\leq n\}.$$
\end{itemize}
\end{definition}
\begin{figure}[ht]
\begin{center}
\begin{tikzpicture}[scale=0.8]

\foreach \x in {0,...,5}{
  \foreach \y in {0,...,5}{
    \fill[fill=gray!50](\x,\y)circle(0.4)node{\footnotesize $\mathcal{H}$};
  }
}
\foreach \x in {0,...,4} {
 \foreach \y in {0,...,4} {
  \filldraw(\x+0.3,\y)circle(2pt)--(\x+0.7,\y)circle(2pt);
  \filldraw(\x,\y+0.3)circle(2pt)--(\x,\y+0.7)circle(2pt);
 }
}

\foreach \x in {0,...,4} {
 \filldraw(\x+0.3,5)circle(2pt)--(\x+0.7,5)circle(2pt);
}

\foreach \y in {0,...,4}{
 \filldraw(5,\y+0.3)circle(2pt)--(5,\y+0.7)circle(2pt);
}

\foreach \t in {0,...,5}{
\filldraw[dashed](-0.3,\t)circle(2pt)--(-1,\t);
\filldraw[dashed](5.3,\t)circle(2pt)--(6,\t);
\filldraw[dashed](\t,-0.3)circle(2pt)--(\t,-1);
\filldraw[dashed](\t,5.3)circle(2pt)--(\t,6);
}

\fill[fill=gray!50](8.5,1)circle(1.1)node{\small$(g,\mathcal{H})$};
\filldraw(9.5,1)circle(0.15)node[above right]{$v_{-1}$};
\filldraw(7.5,1)circle(0.15)node[above left]{$v_{1}$};
\filldraw(8.5,2)circle(0.15)node[above left]{$v_{-2}$};
\filldraw(8.5,0)circle(0.15)node[below left]{$v_{2}$};

\fill[fill=gray!50](12.5,1)circle(1.1)node{\small$(gg_1,\mathcal{H})$};
\filldraw(13.5,1)circle(0.15)node[above right]{$v_{-1}$};
\filldraw(11.5,1)circle(0.15)node[above left]{$v_{1}$};
\filldraw(12.5,2)circle(0.15)node[above]{$v_{-2}$};
\filldraw(12.5,0)circle(0.15)node[below]{$v_{2}$};

\fill[fill=gray!50](8.5,5)circle(1.1)node{\small$(gg_2,\mathcal{H})$};
\filldraw(9.5,5)circle(0.15)node[right]{$v_{-1}$};
\filldraw(7.5,5)circle(0.15)node[left]{$v_{1}$};
\filldraw(8.5,6)circle(0.15)node[above left]{$v_{-2}$};
\filldraw(8.5,4)circle(0.15)node[below left]{$v_{2}$};

\draw[thick](9.5,1)--(11.5,1);
\draw(10.5,1)node[below]{$g_1$};

\draw[thick](8.5,2)--(8.5,4);
\draw(8.5,3)node[right]{$g_2$};

\end{tikzpicture} 
\caption{An example of a Cayley type graph obtained by $\ZZ^2$ action on some finite graph $\mathcal{H}$.}
\end{center}
\end{figure}
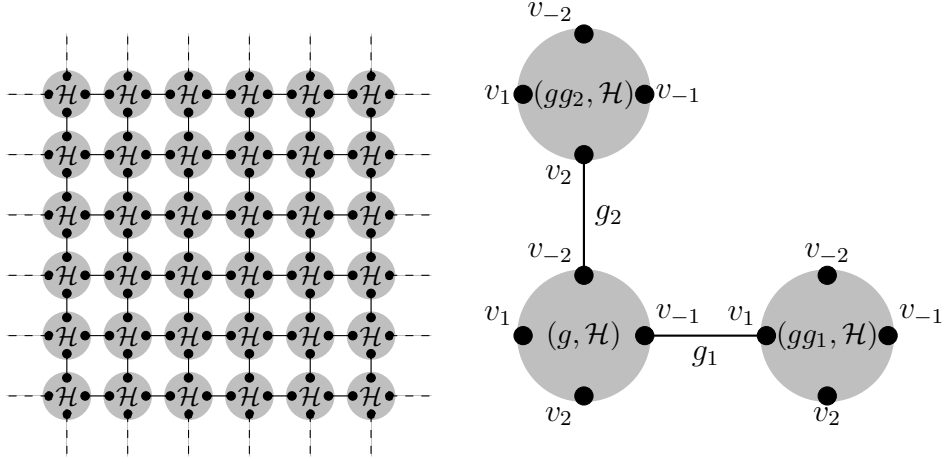
An important fact to note is that the graph $\mathcal{H}_G$ depends on the generator set of $G$. This can easily be demonstrated by focusing on the fact that Cayley graphs for a group may not be isomorphic for different generator sets. So, the graph described above is dependent on $\mathcal{H}$, $G$, $\{v_i\}_{i=-n}^n$ and also $\{g_i\}_{i=1}^n$.

Before moving forward, it should be noted that the vertices $\{v_i\}_{i=-n}^{n}$ in the definition need not be distinct.
So, one can take a tree $\mathcal{T}_l$ with root $e$ (similar to previous section) and set all the $v_i$ to be $e$ and generate the graph $\mathcal{T}_G$.
One should note that the graph $\mathcal{T}_G$, when $G$ is a free group, is not
isomorphic to canopy tree; hence the previous result is not a restriction of this case.

For the graph $\mathcal{H}_G=(\mathcal{V}_G,\mathcal{E}_G)$, we can define the adjacency operator $\Delta_{\mathcal{H}_G}$ on $\ell^2(\mathcal{H}_G)$ by
$$(\Delta_{\mathcal{H}_G}u)((v,h))=\sum_{\{(w,g),(v,h)\}\in \mathcal{E}_G} u((w,g))\qquad\forall (v,h)\in \mathcal{V}_G,$$
and define the projection $P_g$, for $g\in G$, by
$$(P_g u)((v,h))=\left\{\begin{matrix} u((v,g))  & g=h \\ 0 & g\neq h \end{matrix}\right.\qquad \forall (v,h)\in\mathcal{V}_G,$$
for $u\in \ell^2(\mathcal{H}_G)$.
With these definitions in place, we can now define the family of Anderson type operators
\begin{equation}\label{eqErgAndOp}
 H^\omega_G=\Delta_{\mathcal{H}_G}+\sum_{g\in G}\omega_g P_g,
\end{equation}
where $\{\omega_g\}_{g\in G}$ are i.i.d real random variables with common distribution $\mu$. 
If we assume that the support of $\mu$ is bounded, then the operator $H^\omega_G$ is bounded almost surely.

Assuming $\mu$ to be a Borel measure, one can use Kolmogorov construction and
view $\omega_g$ as a random variable over the product probability space $(\RR^{G}, \otimes^G \mathcal{B}_{\RR}, \otimes^G \mu)$ which will be  denoted by $(\Omega,\mathcal{B},\mathbb{P})$. 
For any $g\in G$, define the measure preserving map $\theta_g:\Omega\rightarrow \Omega$ by
$$(\theta_g\omega)_h=\omega_{gh}\qquad\forall h\in G,$$
and the unitary operator $U_g: \ell^2(\mathcal{H}_G)\rightarrow \ell^2(\mathcal{H}_G)$
$$(U_g u)((v,h))= u((v,gh))\qquad\forall (v,h)\in\mathcal{V}_G,$$
and observe that
$$U_g H^\omega_G U_g^\ast =H^{\theta_g(\omega)}_G\qquad\forall g\in G$$
almost surely. Hence the family of random operators $H^\omega_G$ is ergodic under the action of the group $G$.

Before going to the main result of this section, let us first focus on the
unitary maps generated by the automorphisms of the graph $\mathcal{H}_G$.
Since graph automorphisms are bijection of vertex set, an automorphism
$\phi:\mathcal{H}_G\rightarrow\mathcal{H}_G$ produces a unitary map $U_\phi:\ell^2(\mathcal{H}_G)\rightarrow \ell^2(\mathcal{H}_G)$ by
$$(U_\phi u)((v,h))=u(\phi((v,h)))\qquad\forall (v,h)\in\mathcal{V}_G,$$
for $u\in\ell^2(\mathcal{V}_G)$. Since a graph automorphism also provides a bijection of edges, we have
$$\Delta_{\mathcal{H}_G}=U_\phi\Delta_{\mathcal{H}_G} U_\phi^\ast.$$
Let $Aut(\mathcal{H}_G)$ denote the group of all automorphisms of the graph $\mathcal{H}_G$, and let
$$Aut_{And} (\mathcal{H}_G)=\{\phi\in Aut(\mathcal{H}_G): H^\omega_{\mathcal{H}_G}=U_\phi H^\omega_{\mathcal{H}_G} U_\phi^\ast~a.s\}$$
denote the group of automorphisms which fix the operator \eqref{eqErgAndOp}. The next theorem will characterize the group $Aut_{And}(\mathcal{H}_G)$.
But first, let us fix a notation, given an undirected graph $\mathcal{H}$ and a set of vertices $V$, we will denote $Aut(\mathcal{H}|V)$ to the group of automorphisms $\phi:\mathcal{H}\rightarrow\mathcal{H}$ satisfying
$$\phi(v)=v\qquad\forall v\in V.$$
With the above notation in place, we can classify the group $Aut_{And}(\mathcal{H}_G)$.
\begin{theorem}\label{thmAut}
Given a finite graph $\mathcal{H}=(\mathcal{V},\mathcal{E})$ along with
vertices $v_{-n}, \ldots, v_{-1}$, $v_1, \ldots, v_n$ and a finitely generated
group $G$ with generators $g_1,\ldots,g_n$ define the graph $\mathcal{H}_G$ by definition \ref{defCay} and the random operator $H^\omega_{\mathcal{H}_G}$ by \eqref{eqErgAndOp} for i.i.d sequence of real random variables $\{\omega_g\}_{g\in G}$ following a continuous distribution $\mu$. Then the map
$$\Theta: \prod_{g\in G}Aut(\mathcal{H}| \{v_{-n},\ldots,v_{-1},v_1,\ldots,v_n\})\rightarrow Aut_{And}(\mathcal{H}_G)$$
defined by
$$\Theta((\phi_g)_{g\in G})((v,h))=(\phi_h(v),h)\qquad\forall (v,h)\in \mathcal{V}_G,$$
for any $(\phi_g)_{g\in G}\in \prod_{g\in G}Aut(\mathcal{H}| \{v_{-n},\ldots,v_{-1},v_1,\ldots,v_n\})$, is a group isomorphism.
\end{theorem}
\begin{proof}
The definition of $\Theta$ automatically implies that $\Theta((\phi_g)_{g\in G})$ is an element of $Aut_{And}(\mathcal{H}_G)$ for any  $(\phi_g)_{g\in G}\in\prod_{g\in G} Aut(\mathcal{H}|\{v_{-n},\ldots,v_{-1},v_1,\ldots,v_n\})$.
The mapping is a group homomorphism is also clear. We only need to show that it
is a bijection. Clearly, the map is an injection; so we only need to show that
it is a surjection.

Let $\psi\in Aut_{And}(\mathcal{H}_G)$, then for any $u:\mathcal{V}_G\rightarrow\CC$ with $\supp(u)\subset \mathcal{V}\times \{g\}$ for some $g\in G$, observe that 
\begin{align*}
0&=[(H^\omega_{\mathcal{H}_G}-U_\psi H^\omega_{\mathcal{H}_G} U_\psi^\ast)u]((v,g))\\
&=(\omega_g-\omega_{\pi(\psi((v,g)))}) u((v,g)),
\end{align*}
for any $v\in\mathcal{H}$, where $\pi:\mathcal{V}_G\rightarrow G$ is the map $\pi((v,h))= h$ for any $(v,h)\in\mathcal{V}_G$. So, we get
$$\pi(\psi((v,g)))=g\qquad\forall (v,g)\in\mathcal{V}_G,$$
which implies $\psi$ restricted to $\mathcal{V}\times\{g\}$ is a bijection and so is a graph isomorphism, for any $g\in G$.
Now let us focus on the edge $\{(v_{-i},g),(v_{i},gg_i)\}\in\mathcal{E}_G$. Note
that since $\psi$ is a graph automorphism, we have
$\{\psi((v_{-i},g)),\psi((v_i,gg_i))\}\in\mathcal{E}_G$ , which by the above argument implies $\psi((v_{j},h))=(v_{j},h)$ for any $h\in G$ and $i\in\{-n,\ldots,-1,1,\ldots,n\}$. 
This gives us the surjection, completing the proof.

\end{proof}
The above result provides all the unitary operators which fix the operator \eqref{eqErgAndOp} and arise from an automorphism of the graph $\mathcal{H}_G$.
The main reason to state the above theorem is because, now we can construct a graph $\mathcal{H}_G$ such that $Aut_{And}(\mathcal{H}_G)$ is trivial. 
We will focus on this feature after the following result.
%
%
%
%
In the following theorem, we will show that the operators of the form \eqref{eqErgAndOp} can have higher multiplicity for its pure point spectrum.

\begin{theorem}\label{thmGErg}
Consider a graph $\mathcal{H}$ such that for the adjacency matrix $\Delta_{\mathcal{H}}$, there exists $E_0\in\sigma(\Delta_{\mathcal{H}})$ of multiplicity at least $l\geq 2$.
Suppose there exist orthonormal eigenvectors $\psi_1,\ldots,\psi_l$ for $E_0$ and vertices $x_1,\ldots,x_m\in\mathcal{V}$, for some $m\geq 1$, satisfying
$$\psi_i(x_j)=0\qquad\forall 1\leq i \leq l, 1\leq j\leq m.$$
For any $\pi:\{-n,\ldots,n\}\rightarrow\{1,\ldots,m\}$, 
let $\mathcal{H}_G=(\mathcal{V}_G,\mathcal{E}_G)$ be the graph defined by definition \ref{defCay} using the graph $\mathcal{H}$ with $v_i=x_{\pi(i)}$ for $-n\leq i\leq n$
and the finitely generated group $G$ with generators $g_1,\ldots, g_n$.
Defining the operator $H^\omega_G$ by \eqref{eqErgAndOp}, where the random variables $\{\omega_g\}_{g\in G}$ are i.i.d real random variables following an absolutely continuous distribution $\mu$, we have
$$E_0+\supp(\mu)\subset \sigma_{pp}(H^\omega_G)\qquad a.s,$$
and the maximum multiplicity of point spectrum in $E_0+\supp(\mu)$ is at least $l$.
\end{theorem}
\begin{proof}
The proof follows similar steps as the proof of Theorem \ref{thmCan1}.
Fix a $g\in G$ and define
$$\Psi^{g,i}((v,h))=\left\{\begin{matrix} \psi_i(v) & h=g \\ 0 & h\neq g \end{matrix}\right.\qquad\forall (v,h)\in\mathcal{V}_G,$$
then 
$$[(H^\omega_G-E_0-\omega_g)\Psi^{g,i}]((v,h))=0\qquad\forall h\neq g, v\in\mathcal{V}$$
holds trivially.
This is because, the only way a term like $\Psi^{g,i}((\cdot,g))$ can show up is through the adjacency operator $\Delta_{\mathcal{H}_G}$;
but then it will be $\Psi^{g,i}((x_j,g))$ for some $1\leq j\leq m$, which is zero.
For any $v\in\mathcal{V}$, we have
$$[(H^\omega_G-E_0-\omega_g)\Psi^{g,i}]((v,g))=[(\Delta_{\mathcal{H}} -E_0)\psi_i](v)=0;$$
hence $\{\Psi^{g,i}\}_{i=1}^l$ are eigenvectors of $H^\omega_G$ for the
eigenvalue $E_0+\omega_g$. They are orthonormal by the definition of $\{\psi_i\}_i$; 
hence the  multiplicity of eigenvalue $E_0+\omega_g$ for the operator $H^\omega_G$ is at least $l$. 

Following the above steps, we get that $\{\omega_g+E_0\}_{g\in G}$ are eigenvalues of $H^\omega_G$ with multiplicity at least $l$.
Since $\{\omega_g\}_{g\in G}$ are i.i.d random variables, we have
$$E_0+\supp(\mu)=\overline{\{E_0+\omega_g:g\in G\}}\subset \sigma_{pp}(H^\omega_G),$$
which completes the proof.

\end{proof}
There are many examples of graph $\mathcal{H}$ which satisfies the hypothesis of
the above theorem. The following examples illustrate a constructive mechanism to
create  these types of graphs.
\begin{example}
Here we are constructing a graph $\mathcal{H}=(\mathcal{V},\mathcal{E})$ such that, there exists an eigenvalue $E_0$ for the adjacency matrix $\Delta_{\mathcal{H}}$ 
with multiplicity at least $l$, which satisfies
$$\psi_i(x_j)=0\qquad\forall 1\leq i \leq l, 1\leq j\leq m,$$
where $\{\psi_i\}_{i=1}^l$ are some orthonormal eigenvectors corresponding to
the eigenvalues $E_0$ and $x_1,\ldots,x_m\in\mathcal{V}$.

Given a sequence of finite undirected graphs $\tilde{\mathcal{H}}_i=(\tilde{\mathcal{V}}_i,\tilde{\mathcal{E}}_i)$ for $1\leq i\leq l+m$ such that 
$$E_0\in \bigcap_{i=1}^{l+m} \sigma(\Delta_{\tilde{\mathcal{H}}_i}),$$
let $\{v_{i,j}\}_{j=1}^m$ be in $\tilde{\mathcal{V}}_i$ for each $i$ (we are allowing the case $v_{i,j}=v_{i,k}$ for some $j\neq k$).
For the graph $\mathcal{H}$, define the vertex set to be
$$\mathcal{V}=\{x_j:1\leq j\leq m\}\cup \bigcup_{i=1}^{l+m} \tilde{\mathcal{V}}_i,$$
where $\{x_j\}_j$ are new vertices, and the edge set is defined by
$$\mathcal{E}=\bigcup_{i=1}^{l+m}\left(\tilde{\mathcal{E}}_i\cup\{\{x_j,v_{i,j}\}:1\leq j\leq m\} \right).$$
Now, let $\psi_i$ denote an eigenvector for the eigenvalue $E_0$ for the adjacency matrix $\Delta_{\tilde{\mathcal{H}}_i}$ and define
$$\Psi^{(\alpha)}(w)=\left\{\begin{matrix} \alpha_i\psi_i(w) &
    w\in\tilde{\mathcal{V}}_i~\text{for some }i \\ 0 & w=x_i~\text{ for some }i\end{matrix}\right.\qquad\forall w\in\mathcal{V},$$
where $\{\alpha_i\}_{i=1}^{l+m}$ satisfies
\begin{equation}\label{exRelEq1}
 \sum_{i=1}^{l+m} \alpha_i \psi_i(v_{i,j})=0\qquad\forall 1\leq j\leq m,
\end{equation}
and $\sum_i |\alpha_i|^2=1$. With this definition note that if $w\in\tilde{\mathcal{V}}_i$ for any $i$, then
$$[(\Delta_{\mathcal{H}}-E_0)\Psi^{(\alpha)}](w)=\alpha_i[(\Delta_{\tilde{\mathcal{H}}_i}-E_0)\psi_i](w)=0.$$
And for $w\in\{x_1,\ldots,x_m\}$, we get
$$[(\Delta_{\mathcal{H}}-E_0)\Psi^{(\alpha)}](w)=\sum_{i=1}^{l+m} \alpha_i \psi_i(v_{i,j})=0\qquad\exists 1\leq j\leq m,$$
because of \eqref{exRelEq1}. 
Viewing equations \eqref{exRelEq1} as matrix equation we get that there are at least $l$ orthonormal $(\alpha)$ which satisfy the equations.
Hence, we get all the properties that we desire for $\mathcal{H}$.
\end{example}
In particular, we can look at a special case of a graph $\mathcal{H}$ for which the only automorphism which fixes $x_i$ for each $i$ is identity.
\begin{example}
In the earlier example take the graph $\tilde{\mathcal{H}}_i$ to be
$$\tilde{\mathcal{V}}_i=\{n:1\leq n\leq 2p_i-1\}~~\&~~\tilde{\mathcal{E}}_i=\{\{n,n+1\}:1\leq n<2p_i-1\},$$
where $p_i$ is the $i^{th}$ prime starting from $2$.

\begin{figure}[ht]
\begin{center}
\begin{tikzpicture}
\fill(0,0)circle(2pt)node[left]{$x_{1}$};
\fill(4,0)circle(2pt)node[right]{$x_{2}$};

\filldraw(0,0)--(1,0)circle(2pt)--(2,0)circle(2pt)--(3,0)circle(2pt)--(4,0);
\filldraw(0,0)--(0.666,0.5)circle(2pt)--(1.333,0.85)circle(2pt)--(2,1)circle(2pt)--(2.666,0.85)circle(2pt)--(3.333,0.5)circle(2pt)--(4,0);
\filldraw(0,0)--(0.4,0.65)circle(2pt)--(0.8,1.3)circle(2pt)--(1.2,1.65)circle(2pt)--(1.6,1.9)circle(2pt)--
(2.0,2.0)circle(2pt)--(2.4,1.9)circle(2pt)--(2.8,1.65)circle(2pt)--(3.2,1.3)circle(2pt)--(3.6,0.65)circle(2pt)--(4.0,0.0);
\filldraw(0,0)--(0.285,0.796)circle(2pt)--(0.571,1.469)circle(2pt)--(0.857,2.02)circle(2pt)--(1.151,2.449)circle(2pt)--(1.423,2.755)circle(2pt)--
(1.714,2.938)circle(2pt)--(2.0,3.0)circle(2pt)--(2.285,2.938)circle(2pt)--(2.571,2.755)circle(2pt)--
(2.857,2.449)circle(2pt)--(3.151,2.02)circle(2pt)--(3.423,1.469)circle(2pt)--(3.714,0.796)circle(2pt)--(4.0,0.0)circle(2pt);

\draw(2,0)node[below]{$\mathcal{V}_1$};
\draw(2,1)node[below]{$\mathcal{V}_2$};
\draw(2,2)node[below]{$\mathcal{V}_3$};
\draw(2,3)node[below]{$\mathcal{V}_4$};
\end{tikzpicture}
 \caption{Example of $\mathcal{H}$ for $K=2$.} 
\end{center}
\end{figure}
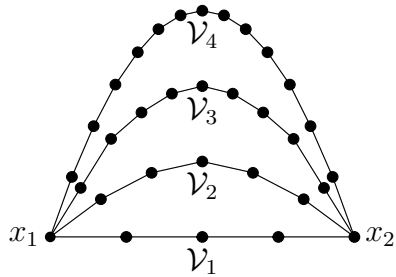
\noindent It is easy to see that 
$$\sigma(\Delta_{\tilde{\mathcal{H}}_i})=\left\{2\cos\frac{\pi j}{2p_i}:1\leq j\leq 2p_i-1\right\},$$
so $\{0\}= \cap_{i=1}^{K+2}\sigma(\Delta_{\tilde{\mathcal{H}}_i})$.
By the construction, it should be clear that any automorphism of $\mathcal{H}$ which fixes $x_1$ and $x_{2}$ is identity. 
Hence, as a consequence of Theorem \ref{thmAut} we get that $Aut_{And}(\mathcal{H}_G)$ is singleton group.
But using this graph in Theorem \ref{thmGErg}, we obtain that the operator $H^\omega_G$ has non-trivial multiplicity.
Hence we conclude that the multiplicity of $H^\omega_{\mathcal{H}_G}$ is not arising from any automorphisms of $\mathcal{H}_G$.

\end{example}
It should be noted that, the above example is nothing special and one can come up with more examples of similar type.
One could have started with $\tilde{\mathcal{V}}_i=\{n:1\leq n\leq 3p_i-1\}$, 
then $\cap_{i=1}^K \sigma(\Delta_{\tilde{\mathcal{H}}_i})=\left\{2\cos \frac{\pi}{3},2\cos \frac{2\pi}{3}\right\}$. 

\bibliographystyle{plain}

\begin{thebibliography}{10}

\bibitem{PA}
Philip~W Anderson.
\newblock Absence of diffusion in certain random lattices.
\newblock {\em Physical review}, 109(5):1492, 1958.

\bibitem{AD1}
M~Anish and Dhriti~Ranjan Dolai.
\newblock Multiplicity theorem of singular spectrum for general anderson type
  hamiltonian.
\newblock {\em arXiv:1709.01774}, 2017.

\bibitem{CL}
Ren{\'e} Carmona and Jean Lacroix.
\newblock {\em Spectral theory of random Schr{\"o}dinger operators}.
\newblock Springer Science \& Business Media, 2012.

\bibitem{DE}
Adrian Dietlein and Alexander Elgart.
\newblock Level spacing for continuum random schr\"{o}dinger operators with
  applications.
\newblock {\em arXiv preprint arXiv:1712.03925}, 2017.

\bibitem{PH}
Peter~D Hislop.
\newblock Lectures on random schr{\"o}dinger operators.
\newblock {\em Fourth Summer School in Analysis and Mathematical Physics},
  476:41--131, 2008.

\bibitem{HK1}
Peter~D Hislop and M~Krishna.
\newblock Eigenvalue statistics for random schr{\"o}dinger operators with non
  rank one perturbations.
\newblock {\em Communications in Mathematical Physics}, 340(1):125--143, 2015.

\bibitem{JL1}
Vojkan Jak\v{s}i\'{c} and Yoram Last.
\newblock Spectral structure of anderson type hamiltonians.
\newblock {\em Inventiones mathematicae}, 141(3):561--577, 2000.

\bibitem{JL2}
Vojkan Jak\v{s}i\'{c} and Yoram Last.
\newblock Simplicity of singular spectrum in anderson-type hamiltonians.
\newblock {\em Duke Mathematical Journal}, 133(1):185--204, 05 2006.

\bibitem{K2}
Werner Kirsch.
\newblock An invitation to random schr{\"o}dinger operators. with an appendix
  by fr{\'e}d{\'e}ric klopp.
\newblock {\em Panor. Synth\'{e}ses. Random Schr\"{o}dinger operators. Soc.
  Math. France, Paris}, pages 1--119, 2008.

\bibitem{KM2}
Abel Klein and Stanislav Molchanov.
\newblock Simplicity of eigenvalues in the anderson model.
\newblock {\em Journal of statistical physics}, 122(1):95--99, 2006.

\bibitem{AM1}
Anish Mallick.
\newblock Jak{\v{s}}i{\'c}-last theorem for higher rank perturbations.
\newblock {\em Mathematische Nachrichten}, 289(11-12):1548--1559, 2016.

\bibitem{AM3}
Anish Mallick.
\newblock Multiplicity bound of singular spectrum for higher rank anderson
  models.
\newblock {\em Journal of Functional Analysis}, 272(12):5162 -- 5190, 2017.

\bibitem{NNS}
Sergey Naboko, Roger Nichols, and G\"{u}nter Stolz.
\newblock Simplicity of eigenvalues in anderson-type models.
\newblock {\em Arkiv f\"{o}r Matematik}, 51(1):157--183, 2013.

\bibitem{npa}
P~A Narayanan.
\newblock Eigenvalue statistics for higher rank anderson model over canopy
  tree.
\newblock {\em arXiv preprint arXiv:1706.02488}, 2017.

\bibitem{SSH}
Christian Sadel and Hermann Schulz-Baldes.
\newblock Random dirac operators with time reversal symmetry.
\newblock {\em Communications in Mathematical Physics}, 295(1):209--242, 2010.

\bibitem{BS2}
Barry Simon.
\newblock Cyclic vectors in the anderson model.
\newblock {\em Reviews in Mathematical Physics}, 6(05a):1183--1185, 1994.

\end{thebibliography}

\end{document}